\documentclass[11pt,a4paper,twoside]{article}

\usepackage{graphicx,url}
\usepackage{setspace}
\usepackage{enumerate,lineno,setspace,float}
\usepackage[small,compact]{titlesec}
\usepackage{amsmath,amsfonts,amssymb,amsthm}
\usepackage{geometry}
\usepackage{fancyhdr}
\usepackage{amssymb,amsmath}
\usepackage{latexsym}
\usepackage{authblk}
\DeclareMathOperator{\lcm}{lcm}
\newtheorem{Theorem}{Theorem}[section]
\newtheorem{Definition}[Theorem]{Definition}

\newtheorem{Corollary}[Theorem]{Corollary}
\newtheorem{Lemma}[Theorem]{Lemma}
\newtheorem{Example}{Example}[section]
\newtheorem{Observation}[Theorem]{Observation}

\theoremstyle{remark}

\usepackage{color}

\definecolor{Blue}{rgb}{0,0,1}
\definecolor{DarkGreen}{rgb}{0,0.6,0}
\definecolor{Red}{rgb}{1,0,0}
\definecolor{Orange}{rgb}{1,0.5,0}

\long\def\delete#1{}

\usepackage{xcolor}
\usepackage[normalem]{ulem}

\input amssym.def
\input amssym.tex

\newcommand{\be}{\begin{equation}}
\newcommand{\ee}{\end{equation}}
\newcommand{\bea}{\begin{eqnarray}}
\newcommand{\eea}{\end{eqnarray}}
\newcommand{\bean}{\begin{eqnarray*}}
\newcommand{\eean}{\end{eqnarray*}}

\def\ra{\rangle}

\def\diam{{\rm diam}}
\def\span{{\rm span}}

\def\rn{{\rm rn}}

\def\ve{\varepsilon}

\def\({\left(}
\def\){\right)}
\def\[{\left[}
\def\]{\right]}
\def\ra{\rightarrow}

\begin{document}

\title{Radio number for the Cartesian product of a tree and a complete graph}
\author[a]{Payal Vasoya \thanks{E-mail:\texttt{prvasoya92@gmail.com} (Payal Vasoya)}}

\affil[a]{Gujarat Technogical University, Ahmedabad - 382 424, Gujarat, India}

\author[b]{Devsi Bantva\thanks{E-mail:\texttt{devsi.bantva@gmail.com} (Devsi Bantva)}\thanks{Corresponding author.}}

\affil[b]{Lukhdhirji Engineering College, Morvi 363 642, Gujarat, India}

\pagestyle{myheadings}
\markboth{\centerline{Payal Vasoya and Devsi Bantva}}{\centerline{Radio number for the Cartesian product of a tree and a complete graph}}

\date{}
\openup 0.8\jot
\maketitle

\begin{abstract}
A radio labelling of a graph $G$ is a mapping $f : V(G) \rightarrow \{0, 1, 2,\ldots\}$ such that $|f(u)-f(v)|\geq \diam(G) + 1 - d(u,v)$ for every pair of distinct vertices $u,v$ of $G$, where $\diam(G)$ is the diameter of $G$ and $d(u,v)$ is the distance between $u$ and $v$ in $G$. The radio number $\rn(G)$ of $G$ is the smallest integer $k$ such that $G$ admits a radio labelling $f$ with $\max\{f(v):v \in V(G)\} = k$. In this paper, we give a lower bound for the radio number of the Cartesian product of a tree and a complete graph and give two necessary and sufficient conditions to achieve the lower bound. We also give three sufficient conditions to achieve the lower bound. We determine the radio number for the Cartesian product of a level-wise regular trees and a complete graph which attains the lower bound. The radio number for the Cartesian product of a path and a complete graph derived in \cite{Kim} can be obtained using our results in a short way.


\end{abstract}

\section{Introduction}\label{intro}
The channel assignment problem is the problem of assigning a channel to each transmitter in a radio network such that a set of constraints is satisfied and the span is minimized. The constraints for assigning channels to transmitters are usually determined by the geographic location of the transmitters; closer the location, the stronger the interference might occur. In order to avoid stronger interference, the larger frequency gap between two assigned frequencies must be required. In \cite{Hale}, Hale designed the optimal labelling problem for graphs to deal with this channel assignment problem. In this model, the transmitters are represented by the vertices of a graph, and two vertices are adjacent if the corresponding transmitters are close to each other. Initially, only two levels of interference, namely \emph{avoidable} and \emph{unavoidable}, were considered which inspired Griggs and Yeh \cite{Griggs} to introduce the following concept: An \emph{$L(2,1)$-labelling} of a graph $G$ is a function $f : V(G) \rightarrow \{0, 1, 2, \ldots\}$ such that $|f(u)-f(v)| \geq 2$ if $d(u,v) = 1$ and $|f(u)-f(v)| \geq 1$ if $d(u,v) = 2$. The span of $f$ is defined as $\max\{|f(u)-f(v)|: u, v \in V(G)\}$, and the \emph{$\lambda$-number} (or the $\lambda_{2,1}$-number) of $G$ is the minimum span of an $L(2,1)$-labelling of $G$. The $L(2,1)$-labelling problem has been studied extensively in the past more than two decades, as one can find in the survey articles \cite{Calamoneri,Yeh1}.

Denote the diameter of a graph $G$ by $\diam(G)$ (the diameter of a graph $G$ is $\diam(G) = \max\{d(u,v) : u, v \in V(G)\}$). In \cite{Chartrand1,Chartrand2}, Chartrand \emph{et al.} introduced the concept of the radio labelling problem by extending the condition on distance in $L(2,1)$-labelling from two to the maximum possible distance in a graph - the diameter of a graph.

\begin{Definition}
{\em
A \emph{radio labelling} of a graph $G$ is a mapping $f: V(G) \rightarrow \{0, 1, 2, \ldots\}$ such that the following hold for every pair of distinct vertices $u, v$ of $G$,
\begin{equation}\label{rn:def}
|f(u)-f(v)| \geq \diam(G) + 1 - d(u,v).
\end{equation}
The integer $f(u)$ is called the \emph{label} of $u$ under $f$, and the \emph{span} of $f$ is defined as $\span(f) = \max\{|f(u)-f(v)|: u, v \in V(G)\}$. The \emph{radio number} of $G$ is defined as
$$
\rn(G) = \min_{f} \span(f)
$$
with minimum taken over all radio labellings $f$ of $G$. A radio labelling $f$ of $G$ is called \emph{optimal} if $\span(f) = \rn(G)$.
}
\end{Definition}

Observe that a radio labelling problem is a min-max type optimization problem. Without loss of generality we may always assume that any radio labelling assigns $0$ to some vertex then the span of a radio labelling is equal to the maximum label used. Since $d(u,v) \leq \diam(G)$, any radio labelling always assigns different labels to distinct vertices. Therefore, a radio labelling $f$ of graph $G$ with order $m$ induces the linear order
\be
\label{eqn:ord}
\vec{V}_f : a_{0}, a_{1}, \ldots, a_{m-1}
\ee
of the vertices of $G$ such that
\be
\label{eq:spf}
0 = f(a_{0}) < f(a_{1}) < \ldots < f(a_{m-1}) = \span(f).
\ee
A linear order $a_0,a_1,\ldots,a_{m-1}$ of $V(T)$ is called optimal linear order if it induced by some optimal radio labelling $f$ of $T$.

The determining radio number of graphs is a tough and challenging task. The radio number is known only for very few graphs and the much attention has been paid to special families of graphs. It turns out that even for some special  graph families the problem may be difficult. For example, the radio number of paths was determined by Liu and Zhu in \cite{Liu}, and even for this basic graph family the problem is nontrivial. In \cite{Daphne2,Daphne3}, Liu and Xie determine the radio number for the square graph of paths and cycles. In \cite{Benson}, Benson \emph{et al.} determined the radio number of all graphs of order $n \ge 2$ and diameter $n-2$. In \cite{Daphne1}, Liu gave a lower bound for the radio number of trees and a necessary and sufficient condition for this bound to be achieved. The author also presented a special class of trees, namely spiders, achieving this lower bound. In \cite{Li}, Li \emph{et al.} determined the radio number of complete $m$-ary trees of height $k$, for $k \geq 1, m \geq 2$. In \cite{Tuza}, Hal\'{a}sz and Tuza gave a lower bound for the radio number of level-wise regular trees and proved further that this bound is tight when all the internal vertices have degree more than three. In \cite{Bantva2}, Bantva \emph{et al.} gave a necessary and sufficient condition for the lower bound given in \cite{Daphne1} to be tight along with two sufficient conditions for achieving this lower bound. Using these results, they also determined the radio number of three families of trees in \cite{Bantva2}. In \cite{Bantva3}, Bantva determined the radio number of some trees obtained by applying a graph operation on given trees. In \cite{Saha}, Liu \emph{et al.} improved the lower bound for the radio number of trees. Recently, in \cite{Bantva1}, Bantva and Liu gave a lower bound for the radio number of block graphs and, three necessary and sufficient conditions to achieve the lower bound. The authors gave three other sufficient conditions to achieve the lower bound and also discuss the radio number of line graphs of trees and block graphs in \cite{Bantva1}. Using these results, the authors determine the radio number of extended star of blocks and level-wise regular block graphs in \cite{Bantva1}.

In this paper, we give a lower bound for the radio number of the Cartesian product of a path and a complete graph (Theorem \ref{thm:lower}). We also give two necessary and sufficient conditions (Theorems \ref{thm:ns1} and \ref{thm:ns2}) and three other sufficient conditions (Theorem \ref{thm:suf}) to achieve the lower bound. Using these results, we determine the radio number of the Cartesian product of a level-wise regular tree and a complete graph. The radio number for the Cartesian product of a path and a complete graph given in \cite{Kim} can be obtained using our results in a short way.

\section{Preliminaries}
\label{prel}

We follow \cite{West} for standard graph-theoretic terms and notations. In a graph $G$, the neighbourhood of any $v \in V(G)$ is $N_G(v) = \{u : u \mbox{ is adjacent to }v\}$. The distance between two vertices $u$ and $v$ in $G$, denoted by $d_G(u,v)$, is the length of the shortest path joining $u$ and $v$ in $G$. The diameter of a graph $G$, denoted by $\diam(G)$, is $\diam(G) = \max\{d_G(u,v) : u, v \in V(G)\}$. We drop the suffix in above defined terms when $G$ is clear in the context. A complete graph $K_n$ is a graph on $n$ vertices in which every two vertices are adjacent. A \emph{tree} is a connected acyclic graph. We fix $V(T) = \{u_0,u_1,\ldots,u_{m-1}\}$ for tree $T$ of order $m$ and $V(K_n) = \{v_0,v_1,\ldots,v_{n-1}\}$ throughout the paper. A vertex $v \in V(T)$ is an \emph{internal vertex} of $T$ if it has degree greater than one and is a \emph{leaf} otherwise. In \cite{Bantva2,Daphne1}, the \emph{weight} of $T$ from $v \in V(T)$ is defined as
$$
w_{T}(v) = \sum_{u \in V(T)}d(u,v)
$$
and the weight of $T$ is defined as
$$
w(T) = \min\{w_{T}(v) : v \in V(T)\}.
$$
A vertex $v \in V(T)$ is a \emph{weight center} \cite{Bantva2,Daphne1} of $T$ if $w_{T}(v) = w(T)$. Denote by $W(T)$ the set of weight centers of $T$. In \cite{Daphne1}, the following is proved about $W(T)$.
\begin{Lemma}\cite{Daphne1}\label{lem:wt1} If $r$ is a weight center of a tree $T$. Then each component of $T-r$ contains at most $|V(T)|/2$ vertices.
\end{Lemma}

\begin{Lemma}\cite{Daphne1}\label{lem:wt2} Every tree $T$ has one or two weight centers, and $T$ has two weight centers, say, $W(T)=\{r,r'\}$ if and only if $rr'$ is an edge of $T$ and $T-rr'$ consists of two equal sized components.
\end{Lemma}

A vertex $u$ is called \cite{Bantva2,Daphne1} an \emph{ancestor} of a vertex $v$, or $v$ is a \emph{descendent} of $u$, if $u$ is on the unique path joining a weight center and $v$. Let $u \in V(T) \setminus W(T)$ be a vertex adjacent to a weight center $x$. The subtree of $T$ induced by $u$ and all its descendants is called the \emph{branch} of $T$ at $u$. Two vertices $u, v$ of $T$ are said to be in \emph{different branches} if the path between them consists only one weight center and in \emph{opposite branches} if the path joining them consists two weight centers. We view a tree $T$ rooted at its weight center $W(T)$: if $W(T) = \{r\}$, then $T$ is rooted at $r$; if $W(T) = \{r,r'\}$, then $T$ is rooted at $r$ and $r'$ in the sense that both $r$ and $r'$ are at level 0. In either case, for each $u \in V(T)$, define
$$
L(u) = \mbox{min}\{d(u,x) : x \in W(T)\}
$$
to indicate the \emph{level} of $u$ in $T$, and define
$$
L(T) = \sum_{u \in V(T)} L(u)
$$
the \emph{total level} of $T$. For any $u, v \in V(T)$, define
\begin{equation*}
\label{eq:phi}
\phi(u,v) = \mbox{max}\{L(x) : x \mbox{ is a common ancestor of $u$ and $v$}\},
\end{equation*}
\begin{equation*}
\label{eq:delta}
\delta(u,v) =
\begin{cases}
1, & \mbox{if $|W(T)| = 2$ and the $(u, v)$-path in $T$ contains both weight centers}, \\
0, & \mbox{otherwise}.
\end{cases}
\end{equation*}

\begin{Lemma}[{\cite[Lemma 2.1]{Bantva2}}]\label{lem4} Let $T$ be a tree with diameter $d \geq 2$. Then for any $u, v \in V(T)$ the following hold:
\begin{enumerate}[\rm (a)]
  \item $\phi(u,v) \geq 0$;
  \item $\phi(u,v) = 0$ if and only if $u$ and $v$ are in different or opposite branches;
  \item $\delta(u,v) = 1$ if and only if $T$ has two weight centers and $u$ and $v$ are in opposite branches;
  \item the distance $d(u,v)$ in $T$ between $u$ and $v$ can be expressed as
  \begin{equation}\label{eqn:dist}
  d(u,v) = L(u) + L(v) - 2\phi(u,v)+\delta(u,v).
  \end{equation}
\end{enumerate}
\end{Lemma}

Define
\begin{equation*}
\ve(T) =
\begin{cases}
1, & \mbox{if $|W(T)| = 1$}, \\
0, & \mbox{if $|W(T)| = 2$}.
\end{cases}
\end{equation*}

Let $G=(V(G),E(G))$ and $H=(V(H),E(H))$ be two graphs. The Cartesian product of $G$ and $H$ is the graph $G \Box H$ with $V(G \Box H) = V(G) \times V(H)$ and two vertices $(a,b)$ and $(c,d)$ are adjacent if $a=c$ and $(b,d) \in E(H)$ or $b=d$ and $(a,c) \in E(G)$. It is clear from the definition that $d_{G \Box H} ((x_1, y_1), (x_2, y_2)) = d_{G} (x_1, x_2) + d_{H} (y_1, y_2)$.

\begin{Observation}\label{Obs1} For a tree $T$ of order $m\,(m \geq 2)$ and a complete graph $K_n$,
\begin{enumerate}[\rm (a)]
\item $|V(T \Box K_n)| = mn$,
\item $\diam(T \Box K_n) = \diam(T)+1$.
\end{enumerate}
\end{Observation}

Let $\vec{V} = (w_0,w_1,\ldots,w_{mn-1})$ be an ordering of $V(T \Box K_n)$, where $|T| = m$. Then note that each $w_t\; (0 \leq t \leq mn-1)$ is an ordered pair $(u_{i_t},v_{j_t})$ with $u_{i_t} \in V(T)$ and $v_{j_t} \in V(K_n)$. Hence each vertex $u_i,\;i=0,1,\ldots,m-1$ appear $n$ times and each vertex $v_j,\;j=0,1,\ldots,n-1$ appear $m$ times in $T \Box K_n$. For $w_a = (u_{i_a},v_{j_a}), w_b = (u_{i_b},v_{j_b}) \in V(T \Box K_n)\;(0 \leq a,b \leq mn-1)$, $v_{j_a}$ and $v_{j_b}$ are called \emph{distinct} if $j_a \neq j_b$. For any $w_a = (u_{i_a},v_{j_a}), w_b = (u_{i_b},v_{j_b}) \in V(T \Box K_n)\;(0 \leq a,b \leq mn-1)$, the distance between $w_a$ and $w_b$ is given by
\begin{equation}\label{dist:TKn}
d(w_a,w_b) = L(u_{i_a})+L(u_{i_b})+\delta(u_{i_a},u_{i_b})-2\phi(u_{i_a},u_{i_b})+d(v_{j_a},v_{j_b})
\end{equation}

\section{A tight lower bound for $\rn(T \Box K_n)$}\label{main:sec}
In this section, we continue to use terms and notations defined in the previous section. We first give a lower bound for the radio number of the Cartesian product of a tree and a complete graph. Next we give two necessary and sufficient conditions and three other sufficient conditions to achieve the lower bound.

\begin{Theorem}\label{thm:lower} Let $T$ be a tree of order $m$ and diameter $d \geq 2$. Denote $\ve = \ve(T)$. Then
\begin{equation}\label{rn:lower}
\rn(T \Box K_n) \geq (mn-1)(d+\ve)-2nL(T).
\end{equation}
\end{Theorem}
\begin{proof} It is enough to prove that any radio labelling of $T \Box K_n$ has no span less than that the right-hand side of \eqref{rn:lower}. Suppose $f$ be any radio labelling of $T \Box K_n$ which induces an ordering $\vec{V}_f = (w_0,w_1,\ldots,w_{mn-1})$ such that $0 = f(w_0) < f(w_1) < \ldots < f(w_{mn-1})$. Denote $\diam(T \Box K_n)=d'$. By the definition of radio labelling, $f(w_{t+1})-f(w_t) \geq d'+1-d(w_t,w_{t+1})$ for $0 \leq t \leq mn-2$. Since $d' = d+1$ by Observation \ref{Obs1}, we have $f(w_{t+1})-f(w_t) \geq d+2-d(w_t,w_{t+1})$ for $0 \leq t \leq mn-2$. Summing up these $mn-1$ inequalities, we obtain
\begin{equation}\label{eqn:spn1}
\span(f) = f(w_{mn-1}) \geq (mn-1)(d+2)-\sum_{t=0}^{mn-2}d(w_t,w_{t+1}).
\end{equation}

\textsf{Case-1:} $|W(T)|=1$. In this case, $\delta(u_{i_t},u_{i_{t+1}}) = 0$ by the definition of $\delta$, $\phi(u_{i_t},u_{i_{t+1}}) \geq 0$ and $0 \leq d(v_{j_t},v_{j_{t+1}}) \leq 1$ for all $0 \leq t \leq mn-2$. Hence, using \eqref{dist:TKn}, we obtain
\begin{eqnarray*}
  \sum_{t=0}^{mn-2}d(w_t,w_{t+1}) &=& \sum_{t=0}^{mn-2}[L(u_{i_t})+L(u_{i_{t+1}})-2\phi(u_{i_t},u_{i_{t+1}})+d(v_{j_t},v_{j_{t+1}})] \\
   & \leq & \sum_{t=0}^{mn-2}[L(u_{i_t})+L(u_{i_{t+1}})+d(v_{j_t},v_{j_{t+1}})] \\
   & \leq & 2n\sum_{i=0}^{m-1}L(u_i)+(mn-1)-L(u_{i_0})-L(u_{i_{mn-1}}) \\
   & \leq & 2nL(T)+mn-1.
\end{eqnarray*}
Substituting this in \eqref{eqn:spn1}, we have $\span(f) \geq (mn-1)(d+1)-2nL(T)$.

\textsf{Case-2:} $|W(T)|=2$. In this case, $0 \leq \delta(u_{i_t},u_{i_{t+1}}) \leq 1$, $\phi(u_{i_t},u_{i_{t+1}}) \geq 0$ and $0 \leq d(v_{j_t},v_{j_{t+1}}) \leq 1$ for all $0 \leq t \leq mn-2$. Hence, using \eqref{dist:TKn}, we obtain
\begin{eqnarray*}
  \sum_{t=0}^{mn-2}d(w_t,w_{t+1}) &=& \sum_{t=0}^{mn-2}[L(u_{i_t})+L(u_{i_{t+1}})+\delta(u_{i_t},u_{i_{t+1}})-2\phi(u_{i_t},u_{i_{t+1}})\\
  & & \hspace{7cm} +d(v_{j_t},v_{j_{t+1}})] \\
   & \leq & \sum_{t=0}^{mn-2}[L(u_{i_t})+L(u_{i_{t+1}})+1+d(v_{j_t},v_{j_{t+1}})] \\
   & \leq & 2n\sum_{i=0}^{m-1}L(u_i)+(mn-1)+(mn-1)-L(u_{i_0})-L(u_{i_{mn-1}}) \\
   & \leq & 2nL(T)+2(mn-1).
\end{eqnarray*}
Substituting this in \eqref{eqn:spn1}, we have $\span(f) \geq (mn-1)d-2nL(T)$.
\end{proof}

\begin{Theorem}\label{thm:ns1} Let $T$ be a tree of order $m$ and diameter $d \geq 2$. Denote $\ve = \ve(T)$. Then
\begin{equation}\label{rn:ns1}
\rn(T \Box K_n) = (mn-1)(d+\ve-2nL(T)
\end{equation}
if and only if there exist an ordering $w_0,w_1,\ldots,w_{mn-1}$ of $V(T \Box K_n)$ such that the following hold:
\begin{enumerate}[\rm (a)]
\item $L(u_{i_0}) = L(u_{i_{mn-1}}) = 0$,
%\item $u_{i_t}$ and $u_{i_{t+1}}$ are in different branches when $|W(T)|=1$ and opposite branches when $|W(T)|=2$ for all $0 \leq t \leq mn-2$,
\item $v_{j_t}$ and $v_{j_{t+1}}$ are distinct for all $0 \leq t \leq mn-2$,
\item For $w_a = (u_{i_a},v_{j_a}), w_b = (u_{i_b},v_{j_b})\,(0 \leq a < b \leq mn-1)$, the distance between any two vertices $u_{i_a}$ and $u_{i_b}$ satisfies
\begin{equation}\label{eqn:dab}
    d(u_{i_a},u_{i_b}) \geq \sum_{t=a}^{b-1}[L(u_{i_t})+L(u_{i_{t+1}})-(d+\ve)]+(d+1).
  \end{equation}
\end{enumerate}
Moreover, under these conditions (a)-(c), the mapping $f$ defined by
\begin{equation}\label{eqn:f00}
f(w_0) = 0
\end{equation}
\begin{equation}\label{eqn:f11}
f(w_{t+1}) = f(w_t)+d+\ve-L(u_{i_t})-L(u_{i_{t+1}}), 0 \leq t \leq mn-2
\end{equation}
is an optimal radio labelling of $T \Box K_n$.
\end{Theorem}
\begin{proof} \textsf{Necessity:} Suppose that \eqref{rn:ns1} holds. Note that \eqref{rn:ns1} holds if equality hold in \eqref{rn:def} and everywhere in the proof of Theorem \ref{thm:lower}. Again observe that the equality holds everywhere in the proof of Theorem \ref{thm:lower} then an ordering $w_0,w_1,\ldots,w_{mn-1}$ of $V(T \Box K_n)$ induced by $f$ satisfies the followings: (1) $L(u_{i_0})= L(u_{i_{mn-1}})=0$, (2) $u_{i_t}$ and $u_{i_{t+1}}$ are in different branches when $|W(T)|=1$ and in opposite branches when $|W(T)|=2$ for all $0 \leq t \leq mn-2$, (3) $v_{j_t}$ and $v_{j_{t+1}}$ are distinct for all $0 \leq t \leq mn-2$. Hence in this case the definition of radio labelling $f$ can be written as $f(w_0) = 0$ and $f(w_{t+1}) = f(w_t)+d'+\ve-L(u_{i_t})-L(u_{i_{t+1}})-1 = f(w_t)+d+\ve-L(u_{i_t})-L(u_{i_{t+1}})$ for $0 \leq t \leq mn-2$, where $d' = \diam(T \Box K_n)$. Summing this last equality for $w_a$ to $w_b\,(0 \leq a < b \leq mn-1)$, we obtain
\begin{equation*}
  f(w_b)-f(w_a) = \sum_{t=a}^{b-1}[d+\ve-L(u_{i_t})-L(u_{i_{t+1}})]
\end{equation*}
Since $f$ is a radio labelling of $T \Box K_n$, we have $f(w_b)-f(w_a) \geq d'+1-d(w_a,w_b) = d+2-d(w_a,w_b)$. Also note that $d(w_a,w_b) = d(u_{i_a},u_{i_b})+1$ by \eqref{dist:TKn}. Substituting these in the above equation, we obtain
\begin{equation*}
d(u_{i_a},u_{i_b}) \geq \sum_{t=a}^{b-1}[L(u_{i_t})+L(u_{i_{t+1}})-(d+\ve)]+(d+1).
\end{equation*}

\textsf{Sufficiency:} Suppose there exist an ordering ($w_0$, $w_1$, $\ldots$, $w_{mn-1}$) of $V(T \Box K_n)$ satisfies the conditions (a)-(c) and $f$ is defined by \eqref{eqn:f00} and \eqref{eqn:f11}. Note that it is enough to prove that $f$ is a radio labelling of $T \Box K_n$ and the span of $f$ is the right-hand side of \eqref{rn:ns1}. Denote $\diam(T \Box K_n) = d'$. Let $w_a$ and $w_b\;(0 \leq a < b \leq mn-1)$ be two arbitrary vertices.
\begin{eqnarray*}
  f(w_b)-f(w_a) &=& \sum_{t=a}^{b-1}(f(w_{t+1})-f(w_t)) \\
   &=& \sum_{t=a}^{b-1}(d+\ve-L(u_{i_t})-L(u_{i_{t+1}})) \\
   &=& d+1-d(u_{i_a},u_{i_b}) \\
   &=& d+2-(d(u_{i_a},u_{i_b})+1) \\
   &=& d'+1-d(w_a,w_b).
\end{eqnarray*}
The span of $f$ is
\begin{eqnarray*}
  \span(f) &=& f(w_{mn-1})-f(w_0) \\
   &=& \sum_{t=0}^{mn-2}(f(w_{t+1})-f(w_t)) \\
   &=& \sum_{t=0}^{mn-2}(d+\ve-L(u_{i_t})-L(u_{i_{t+1}})) \\
   &=& (mn-1)(d+\ve)-2\sum_{t=0}^{mn-2}L(u_{i_t})+L(u_{i_0})+L(u_{i_{mn-1}}) \\
   &=& (mn-1)(d+\ve)-2nL(T).
\end{eqnarray*}
\end{proof}

\begin{Theorem}\label{thm:ns2} Let $T$ be a tree of order $m$ and diameter $d \geq 2$. Denote $\ve = \ve(T)$. Then
\begin{equation}\label{rn:ns2}
\rn(T \Box K_n) = (mn-1)(d+\ve)-2nL(T)
\end{equation}
if and only if there exist an ordering $w_0,w_1,\ldots,w_{mn-1}$ with $w_t = (u_{i_t},v_{j_t}), 0 \leq t \leq mn-1$ of $V(T \Box K_n)$ such that the following all hold:
\begin{enumerate}[\rm (a)]
\item $L(u_{i_0}) = L(u_{i_{mn-1}}) = 0$,
\item $v_{j_t}$ and $v_{j_{t+1}}$ are distinct for $0 \leq t \leq mn-2$,
\item $u_{i_t}$ and $u_{i_{t+1}}$ are in different branches when $|W(T)|=1$ and in opposite branches when $|W(T)|=2$ for $0 \leq t \leq mn-2$,
\item $L(u_{i_t}) \leq (d+1)/2$ when $|W(T)|=1$ and $L(u_{i_t}) \leq (d-1)/2$ when $|W(T)|=2$ for all $0 \leq t \leq mn-1$;
\item For any $w_a = (u_{i_a},v_{j_a}), w_b = (u_{i_b},v_{j_b})\,(0 \leq a < b \leq mn-1)$ such that $u_{i_a}$ and $u_{i_b}$ are in the same branch of $T$ then $u_{i_a}$ and $u_{i_b}$ satisfies
\begin{equation}\label{phiab}
\phi(u_{i_a},u_{i_b}) \leq \(\frac{b-a-1}{2}\)(d+\ve)-\sum_{t=a+1}^{b-1}L(u_{i_t})-\(\frac{1-\ve}{2}\).
\end{equation}
\end{enumerate}
\end{Theorem}
\begin{proof} \textsf{Necessity:} Suppose \eqref{rn:ns2} holds. Then there exist an ordering ($w_0$, $w_1$, $\ldots$, $w_{mn-1}$) of $T \Box K_n$ such that the conditions (a)-(c) of Theorem \ref{thm:ns1} holds. Hence the conditions (a) and (b) satisfies. Taking $a=t$ and $b=t+1$ in \eqref{eqn:dab}, we obtain $d(w_t,w_{t+1}) = d(u_{i_t},u_{i_{t+1}})+d(v_{j_t},v_{j_{t+1}}) = L(u_{i_t})+L(u_{i_{t+1}})+1-\ve+1$. Hence we have $d(u_{i_t},u_{i_{t+1}}) = L(u_{i_t})+L(u_{i_{t+1}})+1-\ve$ for all $0 \leq t \leq mn-2$ as $0 \leq d(u_{i_t},u_{i_{t+1}}) \leq L(u_{i_t})+L(u_{i_{t+1}})+1-\ve$ and $0 \leq d(v_{j_t},v_{j_{t+1}}) \leq 1$. Therefore, by Lemma \ref{lem4}, $u_{i_t}$ and $u_{i_{t+1}}$ are in different branches when $|W(T)|=1$ and in opposite branches when $|W(T)|=2$ for all $0 \leq t \leq mn-2$. Hence the condition (c) satisfies. Since $L(u_{i_0}) = L(u_{i_{mn-1}}) = 0$, it is clear that $L(u_{i_t}) < (d+1)/2$ for $t=0,mn-1$. For $1 \leq t \leq mn-2$, consider $u_{i_{t-1}}$ and $u_{i_{t+1}}$ in \eqref{eqn:dab} and using \eqref{dist:TKn}, we obtain
\begin{equation*}
2L(u_{i_t}) \leq (d+\ve)-(1-\ve)+\delta(u_{i_{t-1}},u_{i_{t+1}})-2\phi(u_{i_{t-1}},u_{i_{t+1}}).
\end{equation*}
In above equation, observe that when $|W(T)|=1$ then $\delta(u_{i_{t-1}},u_{i_{t+1}})=0$ by the definition of $\delta$ and when $|W(T)|=2$ then $u_{i_{t-1}}$ and $u_{i_{t+1}}$ are in different or in the same branch of $T$ and hence $\delta(u_{i_{t-1}},u_{i_{t+1}})=0$. Thus we have,
\begin{equation*}
2L(u_{i_t}) \leq d+2\ve-1.
\end{equation*}
Hence, the condition (d) satisfies. To prove (e), let $w_a = (u_{i_a},v_{j_a})$ and $w_b = (u_{i_b},v_{j_b}) (0 \leq a < b \leq mn-1)$ be such that $u_{i_a}$ and $u_{i_b}$ are in the same branch of $T$. Then $\delta(u_{i_a},u_{i_b})=0$. Using \eqref{dist:TKn} in \eqref{eqn:dab}, the \eqref{phiab} can be obtain easily from \eqref{eqn:dab}.

\textsf{Sufficiency:} Suppose there exist an ordering $(w_0,w_1,\ldots,w_{mn-1})$ of $V(T \Box K_n)$ such that conditions (a)-(e) holds. To prove \eqref{rn:ns2} holds, we show that an ordering $w_0,w_1,\ldots,w_{mn-1}$ satisfies the condition (a)-(c) of Theorem \ref{thm:ns1}. Since the conditions (a) and (b) are identical with the conditions (a) and (b) of Theorem \ref{thm:ns1}, we need to prove the condition \eqref{eqn:dab} only. Denote the right-hand side of \eqref{eqn:dab} by $S_{a,b}$. Let $w_a = (u_{i_a},v_{j_a}), w_b = (u_{i_b},v_{j_b})\; (0 \leq a < b \leq mn-1)$ be any two vertices. If $u_{i_a}$ and $u_{i_b}$ are in opposite branches then $d(u_{i_a},u_{i_b}) = L(u_{i_a})+L(u_{i_b})+1$. Hence, $S_{a,b} = L(u_{i_a})+L(u_{i_b})+2\sum_{t=a+1}^{b-1}L(u_{i_t})-(b-a)d+d+1 \leq L(u_{i_a})+L(u_{i_b})+1+2(b-a-1)((d-1)/2)-(b-a-1)d = L(u_{i_a})+L(u_{i_b})+1-(b-a-1) \leq L(u_{i_a})+L(u_{i_b})+1 = d(u_{i_a},u_{i_b})$. If $u_{i_a}$ and $u_{i_b}$ are in different branches then $d(u_{i_a},u_{i_b}) = L(u_{i_a})+L(u_{i_b})$ and $S_{a,b} = L(u_{i_a})+L(u_{i_b})+2\sum_{t=a+1}^{b-1}L(u_{i_t})-(b-a-1)(d+\ve)+(1-\ve)$. If $|W(T)|=1$ then note that $L(u_{i_t}) \leq (d+1)/2$ for all $0 \leq t \leq mn-1$ and hence $S_{a,b} \leq L(u_{i_a})+L(u_{i_b})+2(b-a-1)((d+1)/2)-(b-a-1)(d+1) = L(u_{i_a})+L(u_{i_b}) = d(u_{i_a},u_{i_b})$. If $|W(T)|=2$ then note that $b-a-1 \geq 1$ and $L(u_{i_t}) \leq (d-1)/2$ for all $0 \leq t \leq mn-1$ and hence $S_{a,b} \leq L(u_{i_a})+L(u_{i_b})+2(b-a-1)((d-1)/2)-(b-a-1)d+1 = L(u_{i_a})+L(u_{i_b})+1-(b-a-a) \leq L(u_{i_a})+L(u_{i_b}) = d(u_{i_a},u_{i_b})$. If $u_{i_a}$ and $u_{i_b}$ are in the same branch then \eqref{eqn:dab} can be obtain easily using \eqref{phiab} which complete the proof.
\end{proof}

\begin{Theorem}\label{thm:suf} Let $T$ be a tree of order $m$ and diameter $d \geq 2$. Denote $\ve = \ve(T)$. Then
\begin{equation}\label{rn:suf}
\rn(T \Box K_n) = (mn-1)(d+\ve)-2nL(T)
\end{equation}
if there exist an ordering $w_0,w_1,\ldots,w_{mn-1}$ of $V(T \Box K_n)$ such that
\begin{enumerate}[\rm (a)]
\item $L(u_{i_0}) = L(u_{i_{mn-1}}) = 0$,
\item $v_{j_t}$ and $v_{j_{t+1}}$ are distinct for all $0 \leq t \leq mn-2$,
\item $u_{i_t}$ and $u_{i_{t+1}}$ are in different branches when $|W(T)|=1$ and in opposite branches when $|W(T)|=2$ for all $0 \leq t \leq mn-2$,
\end{enumerate}
and one of the following holds:
\begin{enumerate}[\rm (a)]
\item[\rm (d)] $\min\{d(u_{i_t},u_{i_{t+1}}),d(u_{i_{t+1}},u_{i_{t+2}})\} \leq (d+1-\ve)/2$ for all $0 \leq t \leq mn-3$,
\item[\rm (e)] $d(u_{i_t},u_{i_{t+1}}) \leq (d+1+\ve)/2$ for all $0 \leq t \leq mn-2$,
\item[\rm (f)] For all $0 \leq t \leq mn-1$, $L(u_{i_t}) \leq (d+1)/2$ when $|W(T)|=1$ and $L(u_{i_t}) \leq (d-1)/2$ when $|W(T)|=2$ and, if $w_a = (u_{i_a},v_{j_a})$ and $w_b = (u_{i_b},v_{j_b}) (0 \leq a < b \leq mn-1)$ such that $u_{i_a}$ and $u_{i_b}$ are in the same branch of $T$ then $b-a \geq d$.
\end{enumerate}
\end{Theorem}
\begin{proof} We show that if (a)-(c) and one of the (d)-(f) holds for an ordering $w_0,w_1,\ldots,w_{mn-1}$ such that $w_t = (u_{i_t},v_{j_t}), 0 \leq t \leq mn-1$ then (a)-(e) of Theorem \ref{thm:ns2} satisfies. Since the conditions (a)-(c) are identical in both Theorems \ref{thm:ns2} and \ref{thm:suf}, we need to verify the conditions (d) and (e) of Theorem \ref{thm:ns2} only. Denote the right-hand side of \eqref{phiab} by $P_{a,b}$. We consider the following two cases.

\textsf{Case-1:} $|W(T)|=1$. In this case, recall that $\ve = 1$ and $\delta(u_{i_t},u_{i_{t+1}}) = 0$ for all $0 \leq t \leq mn-2$ by the definition of $\delta$.

\textsf{Subcase-1.1:} Suppose (a)-(d) holds. It is clear that $L(u_{i_0}) = L(u_{i_{mn-1}}) = 0 \leq (d+1)/2$ and for $1 \leq t \leq mn-2$, $L(u_{i_t}) \leq \min\{d(u_{i_t},u_{i_{t+1}}),d(u_{i_{t+1}},u_{i_{t+2}})\} \leq d/2 < (d+1)/2$. Let $w_a = (u_{i_a},v_{j_a})$ and $w_b = (u_{i_b},v_{j_b}),\; 0 \leq a < b \leq mn-1$ be two arbitrary vertices such that $u_{i_a}$ and $u_{i_b}$ are in the same branch of $T$. If $b-a \geq 4$ then $P_{a,b} \geq 3(d+1)/2-d/2-(d+1)/2 = d/2+1 \geq \phi(u_{i_a},u_{i_b})$. Assume $b-a=3$. If $L(u_{i_{a+1}})+L(u_{i_{a+2}}) \leq d/2$ then $P_{a,b} \geq (d+1)-d/2 = d/2+1 \geq \phi(u_{i_a},u_{i_b})$. If $L(u_{i_{a+1}})+L(u_{i_{a+2}}) > d/2$ then $L(u_{i_{a+2}})+L(u_{i_{a+3}}) \leq d/2$ and hence $L(u_{i_{a+2}}) \leq d/2-L(u_{i_{a+3}})$. Therefore, $P_{a,b} \geq d+1-L(u_{i_{a+1}})-L(u_{i_{a+2}}) \geq d+1-(d+1)/2-d/2+L(u_{i_{a+3}}) = L(u_{i_{a+3}})+1/2 \geq \phi(u_{i_a},u_{i_b})$. Assume $b-a = 2$ then either $L(u_{i_a})+L(u_{i_{a+1}}) \leq d/2$ or $L(u_{i_{a+1}})+L(u_{i_{a+2}}) \leq d/2$. Without loss of generality assume that $L(u_{i_a})+L(u_{i_{a+1}}) \leq d/2$ then $L(u_{i_{a+1}}) \leq d/2-L(u_{i_a})$. Hence, $P_{a,b} = (d+1)/2-L(u_{i_{a+1}}) \geq (d+1)/2-d/2+L(u_{i_a}) = L(u_{i_a})+(1/2) \geq \phi(u_{i_a},u_{i_b})$.

\textsf{Subcase-1.2:} Suppose (a)-(c) and (e) holds. Note that $L(u_{i_0}) = L(u_{i_{mn-1}}) = 0 \leq (d+1)/2$. Since $L(u_{i_t}) \geq 1$ and $L(u_{i_t})+L(u_{i_{t+1}}) = d(u_{i_t},u_{i_{t+1}}) \leq (d+2)/2$ for $0 \leq t \leq mn-1$, we obtain $L(u_{i_t}) \leq (d+1)/2$ for all $0 \leq t \leq mn-1$.

Let $w_a = (u_{i_a},v_{j_a})$ and $w_b = (u_{i_b},v_{j_b}), 0 \leq a < b \leq mn-1$ be two arbitrary vertices such that $u_{i_a}$ and $u_{i_b}$ are in the same branch of $T$. If $b-a \geq 3$ then $P_{a,b} \geq d+1-L(u_{i_{a+1}})-L(u_{i_{a+2}}) \geq d+1-(d+2)/2 = d/2 \geq \phi(u_{i_a},u_{i_b})$. Assume $b-a = 2$. Note that $L(u_{i_a})+L(u_{i_{a+1}}) \leq (d+2)/2$ and hence $L(u_{i_{a+1}}) \leq (d+2)/2-L(u_{i_a})$. Therefore, $P_{a,b} = (d+1)/2-L(u_{i_{a+1}}) \geq (d+1)/2-((d+2)/2-L(u_{i_a})) = L(u_{i_a})-(1/2) \geq \phi(u_{i_a},u_{i_b})$.

\textsf{Subcase-1.3:} Suppose (a)-(c) and (f) holds. Assume $d$ is even then $L(u_{i_t}) \leq d/2$ for all $0 \leq t \leq mn-1$ and hence $P_{i,j} \geq ((b-a-1)/2)(d+1)-(b-a-1)(d/2) \geq (b-a-1)/2 \geq (d-1)/2 \geq \phi(u_{i_a},u_{i_b})$. Assume $d$ is odd then note that only for one vertex $u_{i_q}$, $L(u_{i_q}) = (d+1)/2$ from consecutive $b-a$ vertices and for all other vertices $u_{i_t}$, $L(u_{i_t}) \leq d/2$. Hence, $P_{a,b} \geq ((b-a-1)/2)(d+1)-(b-a-2)(d/2)-(d+1)/2 \geq (b-a-2)/2 \geq (d-2)/2 \geq \phi(u_{i_a},u_{i_b})$.

\textsf{Case-2:} $|W(T)|=2$. In this case, recall that $\ve=0$ and $\delta(u_{i_t},u_{i_{t+1}}) = 1$ for all $0 \leq t \leq mn-1$ by the definition of the ordering $w_0,w_1,\ldots,w_{mn-1}$ of $V(T \Box K_n)$.

\textsf{Subcase-2.1:} Suppose (a)-(d) holds. It is clear that $L(u_{i_0}) = L(u_{i_{mn-1}}) = 0 \leq (d-1)/2$ and for $1 \leq t \leq mn-2$, $L(u_{i_t}) \leq \min\{d(u_{i_{t-1}},u_{i_{t}}),d(u_{i_t},u_{i_{t+1}})\}-1 \leq (d+1)/2-1 = (d-1)/2$. Let $w_a = (u_{i_a},v_{j_a})$ and $w_b = (u_{i_b},v_{j_b}),\, 0 \leq a < b \leq mn-1$ be two arbitrary vertices such that $u_{i_a}$ and $u_{i_b}$ are in the same branch of $T$. If $b-a \geq 4$ then $P_{a,b} \geq 3d/2-(d-1)-1/2 = (d+1)/2 \geq \phi(u_{i_a},u_{i_b})$. Assume $b-a = 3$. If $L(u_{i_{a+1}})+L(u_{i_{a+2}}) \leq (d-1)/2$ then $P_{a,b} \geq d-(d-1)/2-1/2 = d/2 \geq \phi(u_{i_a},u_{i_b})$. If $L(u_{i_{a+1}})+L(u_{i_{a+2}}) > (d-1)/2$ then $L(u_{i_a})+L(u_{i_{a+1}}) \leq (d-1)/2$ and hence $L(u_{i_{a+1}}) \leq (d-1)/2-L(u_{i_a})$. Therefore, $P_{a,b} \geq d-((d-1)/2-L(u_{i_a}))-(d-1)/2-1/2 = L(u_{i_a})+1/2 \geq \phi(u_{i_a},u_{i_b})$. Assume $b-a = 2$ then either $L(u_{i_a})+L(u_{i_{a+1}}) \leq (d-1)/2$ or $L(u_{i_{a+1}})+L(u_{i_{a+2}}) \leq (d-1)/2$. Without loss of generality, assume that $L(u_{i_{a+1}})+L(u_{i_{a+2}}) \leq (d-1)/2$ then $L(u_{i_{a+1}}) \leq (d-1)/2-L(u_{i_{a+2}})$. Hence, $P_{a,b} = d/2-L(u_{i_{a+1}})-1/2 \geq d/2-((d-1)/2-L(i_{a+2}))-1/2 = L(u_{i_{a+2}}) \geq \phi(u_{i_a},u_{i_b})$.

\textsf{Subcase-2.2:} Suppose (a)-(c) and (e) holds. Note that $L(u_{i_0}) = L(u_{i_{mn-1}}) = 0 \leq (d-1)/2$. Since $L(u_{i_t}) \geq 1$ and $L(u_{i_t})+L(u_{i_{t+1}}) = d(u_{i_t},u_{i_{t+1}})-1 \leq (d+1)/2-1 = (d-1)/2$ for $0 \leq t \leq mn-1$, we obtain $L(u_{i_t}) \leq (d-1)/2$ for all $0 \leq t \leq mn-1$.

Let $w_a = (u_{i_a},v_{j_a})$ and $w_b = (u_{i_b},v_{j_b}),\; 0 \leq a < b \leq mn-1$ be two arbitrary vertices such that $u_{i_a}$ and $u_{i_b}$ are in the same branch of $T$. If $b-a \geq 3$ then $P_{a,b} \geq d-L(u_{i_{a+1}})-L(u_{i_{a+2}})-1/2 \geq d-(d-1)/2-1/2 = d/2 \geq \phi(u_{i_a},u_{i_b})$. Assume that $b-a = 2$. Then note that $L(u_{i_a})+L(u_{i_{a+1}}) \leq (d-1)/2$ and hence $L(u_{i_a}) \leq (d-1)/2-L(u_{i_{a+1}})$. Hence, $P_{a,b} \geq d/2-L(u_{i_t})-(1/2) \geq d/2-[(d-1)/2-L(u_{i_{a+1}})]-(1/2) = L(u_{i_{a+1}}) \geq \phi(u_{i_a},u_{i_b})$.

\textsf{Subcase-2.3:} Suppose (a)-(c) and (f) holds. Assume $d$ is even then $L(u_{i_t}) \leq (d-2)/2$ for all $0 \leq t \leq mn-1$ and hence $P_{a,b} \geq ((b-a-1)/2)d-(b-a-1)((d-2)/2)-1/2 = b-a-(3/2) \geq d-(3/2) \geq \phi(u_{i_a},u_{i_b})$. Assume $d$ is odd then note that only for one vertex $u_{i_q}$, $L(u_{i_q}) \leq (d-1)/2$ from $b-a$ consecutive vertices and for all other vertices $u_{i_t}$, $L(u_{i_t}) \leq (d-3)/2$. Hence, $P_{a,b} \geq ((b-a-1)/2)d-(b-a-2)((d-3)/2)-(d-1)/2-(1/2) = 3(b-a-2)/2 \geq 3(d-2)/2 \geq \phi(u_{i_a},u_{i_b})$.
\end{proof}

\section{Radio number for the Cartesian product of level-wise regular tree and complete graph}

In this section, using the results of Section \ref{main:sec}, we determine the radio number for the Cartesian product of a level-wise regular tree and a complete graph.

It is well known that the center of tree $T$ consists of one vertex $r$ or two adjacent vertices $r,r'$ depending on $\diam(T)$ is even or odd. We view a tree $T$ rooted at $r$ or at both $r,r'$ respectively. Hal\'{a}sz and Tuza \cite{Tuza} defined a level-wise regular tree to be a tree $T$ in which all vertices at distance $i$ from root $r$ or $\{r,r'\}$ have the same degree, say $d_i$ for $0 \leq i \leq h$, where $h$ (the largest distance from a vertex to the root) is the height of tree $T$. Note that a level-wise regular tree is determined by its center(s) and $(d_0,d_1,\ldots,d_{h-1})$. Denote the level-wise regular tree by $T^z = T^z_{d_0,d_1,\ldots,d_{h-1}}$, where $z$ denotes the number of roots of level-wise regular tree. A star $K_{1,q}$ is a tree consisting of $q$ leaves and another vertex joined to all leaves by edges. A double star $D_q$ is a graph obtained by joining the center vertices of two copies of star graph $K_{1,q}$ by an edge. A banana tree $B_{q,k}$ is a graph constructed by connecting a single leaf from $q$ distinct copies of a star $K_{1,k}$ with a single vertex that is distinct from all the vertices of star graphs. Note that a star $K_{1,q}$, double star $D_q$ and a banana tree $B_{q,k}$ are level-wise regular trees $T^1_q$, $T^2_{q+1}$ and $T^1_{q,2,k}$,  respectively.

\begin{Theorem} Let $h \geq 1$ and $d_i,n \geq 3$ for $0 \leq i \leq h-1$ be any integers. Denote $\ve = \ve(T^z)$. Then $\rn(T^z \Box K_n)$
\begin{equation}\label{rn:level}
=
\left\{
\begin{array}{ll}
\[\(\displaystyle\sum_{i=1}^{h-1}(2h-2i-1)\(\prod_{1 \leq j \leq i}(d_j-1)\)\)+2h-1\]nd_0+(2h+1)(n-1), & \mbox{if $z = 1$}, \\[0.6cm]
\[\(\displaystyle\sum_{i=0}^{h-1}(2h-2i-1)\(\prod_{0 \leq j \leq i}(d_j-1)\)\)+2h+1\]2n-2h-1, & \mbox{if $z = 2$}.
\end{array}
\right.
\end{equation}
\end{Theorem}
\begin{proof} The order of $T^z$, $L(T^z)$ and diameter $d$ of $T^z$ are given by
\begin{equation*}
|V(T^z)| =
\begin{cases}
1+d_0+d_0\displaystyle\sum_{i=1}^{h-1}\(\displaystyle\prod_{1 \leq j \leq i}(d_j-1)\), & \mbox{if $z = 1$}, \\
2+2\displaystyle\sum_{i=0}^{h-1}\(\displaystyle\prod_{0 \leq j \leq i}(d_j-1)\), & \mbox{if $z = 2$},
\end{cases}
\end{equation*}

\begin{equation*}
L(T^z) =
\begin{cases}
d_0+d_0\displaystyle\sum_{i=1}^{h-1}(i+1)\(\displaystyle\prod_{1 \leq j \leq i}(d_j-1)\), & \mbox{if $z = 1$}, \\
2\displaystyle\sum_{i=0}^{h-1}(i+1)\(\displaystyle\prod_{0 \leq j \leq i}(d_j-1)\), & \mbox{if $z = 2$},
\end{cases}
\end{equation*}

\begin{equation*}
d =
\begin{cases}
2h, & \mbox{if $z = 1$}, \\
2h+1, & \mbox{if $z = 2$}.
\end{cases}
\end{equation*}
Substituting the above in \eqref{rn:lower}, we obtain the right-hand side of \eqref{rn:level} as a lower bound for $\rn(T^z \Box K_n)$. We now prove that this lower bound is tight by giving an ordering of $V(T^z \Box K_n)$ which satisfies Theorem \ref{thm:ns1}. For this purpose, denote $V(K_n) = \{v_0,v_1,\ldots,v_{n-1}\}$ and $V(T^z) = \{u_0,u_1,\ldots,u_{m-1}\}$ such that the ordering $u_0,u_1,\ldots,u_{n-1}$ is obtained as follows.

\textsf{Case-1:} $z=1$.

In this case, let $r$ be the unique center of $T^1$. Denote $d_0$ children of $r$ by $r_0,\ldots,r_{d_0-1}$. For $0 \leq i \leq d_0-1$, denote $d_1-1$ children of $r_i$ by $r_{i,0},r_{i,1},\ldots,r_{i,d_1-1}$. Inductively, denote the $d_{l}-1$ children of $r_{i_1,i_2,\ldots,i_l}\;(0 \leq i_1 \leq d_0-1, 0 \leq i_2 \leq d_1-1,\ldots,0 \leq i_l \leq d_{l-1}-1)$ by $r_{i_1,i_2,\ldots,i_l,i_{l+1}}$, where $0 \leq i_{l+1} \leq d_l-1$. Continue this process until all the vertices of $T^1$ got indexed. Now obtain  an ordering $u_0,u_1,\ldots,u_{m-1}$ of vertices of $T^1$ as follows: Set $u_0 := r$ and for $1 \leq t \leq m-1$, set $u_t := r_{i_1,i_2,\ldots,i_l}$, where $j=1+i_1+i_2d_0+\ldots+i_ld_0(d_1-1)\ldots(d_l-1)+\displaystyle\sum_{l+1 \leq t \leq h}d_0(d_1-1)\ldots(d_t-1)$.

\textsf{Case-2:} $z=2$.

In this case, let $r$ and $r'$ be two centers of $T^2$. Denote $d_0-1$ children of $r$ and $r'$ by $r_0,r_1,\ldots,r_{d_0-2}$ and $r'_{0},r'_{1},\ldots,r'_{d_0-2}$, respectively. For $0 \leq i \leq d_0-2$, denote $d_1-1$ children of $r_i$ and $r'_i$ by $r_{i,0},r_{i,1},\ldots,r_{i,d_1-1}$ and $r'_{i,0},r'_{i,1},\ldots,r'_{i,d_1-1}$, respectively. Inductively, denote the $d_l-1$ children of $r_{i_1,i_2,\ldots,i_l}$ and $r'_{i_1,i_2,\ldots,i_l}$, where $0 \leq i_1 \leq d_0-2, 0 \leq i_2 \leq d_1-1,\ldots,0 \leq i_l \leq d_{l-1}-1$ by $r_{i_1,i_2,\ldots,i_l,i_{l+1}}$ and $r'_{i_1,i_2,\ldots,i_l,i_{l+1}}$ respectively, where $0 \leq i_{l+1} \leq d_l-1$. Continue this process until all the vertices of $T^2$ got indexed. Rename $x_j := r_{i_1,i_2,\ldots,i_l} \mbox{ and } x'_j := r'_{i_1,i_2,\ldots,i_l}$,
where $j = 1+i_1+i_2(d_0-1)+\ldots+i_l(d_0-1)(d_1-1)\ldots(d_l-1)+\displaystyle\sum_{l+1 \leq t \leq h}(d_0-1)(d_1-1)\ldots(d_t-1)$. Now obtain an ordering $u_0,u_1,\ldots,u_{m-1}$ as follows: Set $u_0 := r, u_{m-1} := r'$ and for $1 \leq t \leq m-2$, set
\begin{equation*}
u_t :=
\begin{cases}
x_{t/2}, & \mbox{if $t \equiv 0$ (mod $2$)}, \\
x'_{(t+1)/2}, & \mbox{if $t \equiv 1$ (mod $2$)}.
\end{cases}
\end{equation*}

We now consider the following two cases to give an ordering $w_0,w_1,\ldots,w_{mn-1}$ of $V(T \Box K_n) = \{(u_i,v_j) : i = 0, 1, \ldots, m-1, j = 0, 1, \ldots, n-1\}$.

\textsf{Case-1:} $|W(T)|=1$.

We consider the following two subcases to define an ordering $w_0,w_1,\ldots,w_{mn-1}$ of $V(T^1 \Box K_n)$.

\textsf{Subcase 1:} $|V(T^1)| \leq |V(K_n)|$.

In this case, for $0 \leq t \leq mn-1$, define $w_t := (u_i,v_j)$, where $i \in \{0,1,\ldots,m-1\}\; j \in \{0,1,\ldots,n-1\}$ and
$$t := \left\{
\begin{array}{ll}
m(j-i)+i & \mbox{if $j \geq i$ and $j-i \neq n-1$}, \\[0.2cm]
m(n+j-i)+i & \mbox{if $j<i$ and $i-j \neq 1$}, \\[0.2cm]
m(n+j-i)+i-1 & \mbox{if $j < i$ and $i-j = 1$}, \\[0.2cm]
mn-1 & \mbox{if $j-i=n-1$}. \\[0.2cm]
\end{array}
\right.$$

\textsf{Subcase 2:} $|V(T^1)| > |V(K_n)|$.

In this case, for $x\lcm(m,n) \leq t \leq (x+1)\lcm (m,n)-1$, where $0 \leq x \leq \gcd(m,n)-1$, set
$$\alpha_t := \left\{
\begin{array}{ll}
(u_i,v_{j+x}), & \mbox{if $t \equiv i$ (mod $m$), $t \equiv j$ (mod $n$) and $0\leq j\leq n-(x+1)$}, \\[0.2cm]
(u_i,v_{j+x-n}), & \mbox{if $t \equiv i$ (mod $m$), $t \equiv j$ (mod $n$) and $n-x\leq j\leq n-1$}, \\[0.2cm]
\end{array}
\right.$$

Define an ordering $w_0,w_1,\ldots,w_{mn-1}$ of $V(T^1 \Box K_n)$ as follows: For $0 \leq t \leq mn-1$, set
$$w_t := \left\{
\begin{array}{ll}
\alpha_t, & \mbox{if $0 \leq t \leq mn-m-1$}, \\[0.2cm]
\alpha_{t+1}, & \mbox{if $mn-m \leq t \leq mn-2$}, \\[0.2cm]
\alpha_{mn-m}, & \mbox{if $t = mn-1$}.
\end{array}
\right.$$

Then, for above defined ordering, it is clear that (a) $L(u_{i_0}) = L(u_{i_{mn-1}}) = L(u_0) = 0$, (b) $v_{j_t}$ and $v_{j_{t+1}}$ are distinct for $0 \leq t \leq mn-2$, (c) $u_{i_t}$ and $u_{i_{t+1}}$ are in different branches for $0 \leq t \leq mn-2$ and (d) $L(u_{i_t}) \leq d/2$ for all $0 \leq t \leq mn-1$. Hence the conditions (a)-(d) of Theorem \ref{thm:ns2} are satisfies. We now show that the condition (e) of Theorem \ref{thm:ns2} holds as follows.

\textsf{Claim-1:} The above defined ordering $w_0,w_1,\ldots,w_{mn-1}$ of $V(T^1 \Box K_n)$ satisfies \eqref{phiab}.

Let $w_a=(u_{i_a},v_{j_a})$ and $w_b = (u_{i_b},v_{j_b})\;(0 \leq a < b \leq mn-1)$ be two vertices such that $u_{i_a}$ and $u_{i_b}$ are in the same branch when viewed as vertices of $T^1$. Denote the right-hand side of \eqref{phiab} by $P_{a,b}$. In this case, note that $d_0 \geq 3$, $\ve = 1$ and $L(u_{t}) \leq d/2$ for all $0 \leq t \leq mn-1$. Observe that $w_a = (u_{i_a},v_{j_a)}$ and $w_b = (u_{i_b},v_{j_b})$ such that $u_{i_a}$ and $u_{i_b}$ both are in the same branch of $T^2$ then there are two possibilities: (1) $i_a \neq i_b$ (2) $i_a = i_b$. In case of (1), we have $b-a \geq \alpha d_0$, where $\alpha \geq \phi(u_{i_a},u_{i_b})$. Hence, $
P_{a,b} = \((b-a-1)/2\)(d+1)-\sum_{t=a+1}^{b-1}L(u_{i_{t}}) \geq \((b-a-1)/2\)(d+1)-\sum_{t=a+1}^{b-1}(d/2) = \((b-a-1)/2\) = \alpha (d_0-1)/2 \geq  \phi(u_{i_a},u_{i_b})$. In case of (2), note that $b-a \geq m \geq d+1$. Hence, $
P_{a,b} = \((b-a-1)/2\)(d+1)-\sum_{t=a+1}^{b-1}L(u_{i_{t}}) \geq \((b-a-1)/2\)(d+1)-\sum_{t=a+1}^{b-1}(d/2) = \((b-a-1)/2\) = d/2 \geq  \phi(u_{i_a},u_{i_b})$.

\textsf{Case-2:} $|W(T)|=2$

Again we consider the following two subcases to define an ordering $w_0,w_1,\ldots,w_{mn-1}$ of $V(T^2 \Box K_n)$.

\textsf{Subcase 1:} $|V(T)| \leq |V(K_n)|$.

In this case, for $0 \leq t \leq mn-1$, define $w_t := (u_i,v_j)$, where $i \in \{0,1,\ldots,m-1\}, j \in \{0,1,\ldots,n-1\}$ and
$$t := \left\{
\begin{array}{ll}
m(j-i)+i, & \mbox{if $j\geq i$}, \\[0.2cm]
m(n+j-i)+i, & \mbox{if $j<i$}.
\end{array}
\right.$$

\textsf{Subcase 2:} $|V(T)| > |V(K_n)|$.

In this case, define an ordering $w_0,w_1,\ldots,w_{mn-1}$ of $V(T^2 \Box K_n)$ as follows: For $x\lcm(m,n) \leq t \leq (x+1)\lcm (m,n)-1$, where $0 \leq x \leq \gcd(m,n)-1$, set
$$w_t := \left\{
\begin{array}{ll}
(u_i,v_{j+x}), & \mbox{if $t \equiv i$ (mod $m$), $t \equiv j$ (mod $n$) and $0\leq j\leq n-(x+1)$}, \\[0.2cm]
(u_i,v_{j+x-n}), & \mbox{if $t \equiv i$ (mod $m$), $t \equiv j$ (mod $n$) and $n-x\leq j\leq n-1$}, \\[0.2cm]
\end{array}
\right.$$

Then, for above defined ordering, it is clear that (a) $L(u_{i_0}) = L(u_{i_{mn-1}}) = L(u_0) = 0$, (b) $v_{j_t}$ and $v_{j_{t+1}}$ are distinct for $0 \leq t \leq mn-2$, (c) $u_{i_t}$ and $u_{i_{t+1}}$ are in different branches for $0 \leq t \leq mn-2$ and (d) $L(u_{i_t}) \leq (d-1)/2$ for all $0 \leq t \leq mn-1$. Hence the conditions (a)-(d) of Theorem \ref{thm:ns2} are satisfied. Now we prove that the condition (e) of Theorem \ref{thm:ns2} holds as follows.

\textsf{Claim-2:} The above defined ordering $w_0,w_1,\ldots,w_{mn-1}$ of $V(T^2 \Box K_n)$ satisfies \eqref{phiab}.

Let $w_a = (u_{i_a},v_{j_a})$ and $w_b = (u_{i_b},v_{j_b})\;(0 \leq a < b \leq mn-1)$ be two vertices such that $u_{i_a}$ and $u_{i_b}$ are in the same branch when viewed as vertices of $T^2$. Denote the right-hand side of \eqref{phiab} by $P_{a,b}$. In this case, note that $d_0 \geq 3$, $\ve = 0$ and $L(u_t) \leq (d-1)/2$ for all $0 \leq t \leq mn-1$. Observe that  $w_a = (u_{i_a},v_{j_a})$ and $w_b = (u_{i_b},v_{j_b})$ such that $u_{i_a}$ and $u_{i_b}$ both are in the same branch of $T^2$ then there are two possibilities: (1) $i_a \neq i_b$, (2) $i_a = i_b$. In case of (1), we have $b-a \geq \alpha (d_0-1)$, where $\alpha \geq \phi(u_{i_a},u_{i_b})$. Hence, $P_{a,b} = \((b-a-1)/2\)d-\sum_{t=a+1}^{b-1}L(u_{i_t})-1/2 \geq \((b-a-1)/2\)d-\sum_{t=a+1}^{b-1}((d-1)/2)-1/2 = (b-a-2)/2 = \alpha(d_0-1)/2 \geq \phi(u_{i_a},u_{i_b})$. In case of (2), note that $b-a \geq m \geq d+1$. Hence, $P_{a,b} = \((b-a-1)/2\)d-\sum_{t=a+1}^{b-1}L(u_{i_t})-1/2 \geq \((b-a-1)/2\)d-\sum_{t=a+1}^{b-1}((d-1)/2)-1/2 = (b-a-2)/2 = \alpha(d-1)/2 \geq \phi(u_{i_a},u_{i_b})$.
\end{proof}

\begin{Corollary}\label{thm:k1mkn} Let $q \geq 3$ and $n \geq 4$ be any integers. Then
\begin{equation*}\label{rn:k1mkn}
\rn(K_{1,q} \Box K_n) = (q+3)n-3.
\end{equation*}
\end{Corollary}

\begin{Corollary}\label{thm:DabKn} Let $q \geq 2$ and $n \geq 4$ be any integers. Then
\begin{equation*}\label{rn:DabKn}
\rn(D_{q} \Box K_n) = 2n(q+3)-3.
\end{equation*}
\end{Corollary}

\begin{Corollary}\label{thm:BmkKn} Let $k, n \geq 4$ and $q \geq 5$ be any integers. Then
\begin{equation*}\label{rn:BmkKn}
\rn(B_{q,k} \Box K_n) = qn(k+6)+7(n-1).
\end{equation*}
\end{Corollary}

\begin{Example} In Table \ref{tab:K16K7}, a vertex ordering $O(V(K_{1,6} \Box K_7)):=w_0,w_1,\ldots,w_{48}$ and an optimal radio labelling of $K_{1,6} \Box K_7$ are shown.
\end{Example}
\begin{table}[ht!]
\centering
\caption{$\rn(K_{1,6} \Box K_7) = 60$.}\label{tab:K16K7}
\begin{tabular}{|c|l|l|l|l|l|l|l|l|l|l|l|}
\hline
$(u_i,v_j)$ & \multicolumn{1} {|c|}{$v_0$} & \multicolumn{1} {|c|}{$v_1$} & \multicolumn{1} {|c|}{$v_2$} & \multicolumn{1} {|c|}{$v_3$} & \multicolumn{1} {|c|}{$v_4$} & \multicolumn{1} {|c|}{$v_5$} & \multicolumn{1} {|c|}{$v_6$} \\ \hline
    $u_{0}$ & $w_{0} \ra 0$ & $w_{7} \ra 9$ &  $w_{14} \ra 18$ & $w_{21} \ra 27$ & $w_{28} \ra 36$ & $w_{35} \ra 45$ & $w_{48} \ra 60$  \\
    $u_1$ & $w_{42} \ra 53$ & $w_{1} \ra 2$ & $w_{8} \ra 11$ & $w_{15} \ra 20$ & $w_{22} \ra 29$ & $w_{29} \ra 38$ & $w_{36} \ra 47$  \\
    $u_2$ & $w_{37} \ra 48$ & $w_{43} \ra 54$ & $w_{2} \ra 3$ & $w_{9} \ra 12$ & $w_{16} \ra 21$ & $w_{23} \ra 30$ & $w_{30} \ra 39$  \\
    $u_3$ & $w_{31} \ra 40$ & $w_{38} \ra 49$ & $w_{44} \ra 55$ & $w_{3} \ra 4$ & $w_{10} \ra 13$ & $w_{17} \ra 22$ & $w_{24} \ra 31$  \\
    $u_4$ & $w_{25} \ra 32$ & $w_{32} \ra 41$ & $w_{39} \ra 50$ & $w_{45} \ra 56$ & $w_{4} \ra 5$ & $w_{11} \ra 14$ & $w_{18} \ra 23$  \\
    $u_5$ & $w_{19} \ra 24$ & $w_{26} \ra 33$ & $w_{33} \ra 42$ & $w_{40} \ra 51$ & $w_{46} \ra 57$ & $w_{5} \ra 6$ & $w_{12} \ra 15$  \\
    $u_6$ & $w_{13} \ra 16$ & $w_{20} \ra 25$ & $w_{27} \ra 34$ & $w_{34} \ra 43$ & $w_{41} \ra 12$ & $w_{47} \ra 58$ & $w_{6} \ra 7$ \\
    \hline
    \end{tabular}
\end{table}

\begin{Example} In Table \ref{tab:D5K8}, a vertex ordering $O(V(D_5 \Box K_7)):=w_0,w_1,\ldots,w_{83}$ and an optimal radio labelling of $D_5 \Box K_7$ are shown.
\begin{table}[ht!]
\centering
\caption{$\rn(D_5 \Box K_7) = $109.}\label{tab:D5K8}
\begin{tabular}{|c|l|l|l|l|l|l|l|l|l|l|l|l|}
\hline
$(u_i,v_j)$ & \multicolumn{1} {|c|}{$v_0$} & \multicolumn{1} {|c|}{$v_1$} & \multicolumn{1} {|c|}{$v_2$} & \multicolumn{1} {|c|}{$v_3$} & \multicolumn{1} {|c|}{$v_4$} & \multicolumn{1} {|c|}{$v_5$} & \multicolumn{1} {|c|}{$v_6$} \\ \hline
    $u_0$ & $w_{0} \ra 0$ & $w_{36} \ra 48$ & $w_{72} \ra 96$ & $w_{24} \ra 32$ & $w_{60} \ra 80$ & $w_{12} \ra 16$ & $w_{48} \ra 64$ \\
    $u_1$ & $w_{49} \ra 66$ & $w_{1} \ra 2$ & $w_{37} \ra 50$ & $w_{73} \ra 98$ & $w_{25} \ra 34$ & $w_{61} \ra 82$ & $w_{13} \ra 18$ \\
    $u_2$ & $w_{14} \ra 19$ & $w_{50} \ra 67$ & $w_{2} \ra 3$ & $w_{38} \ra 51$ & $w_{74} \ra 99$ & $w_{26} \ra 35$ & $w_{62} \ra 83$ \\
    $u_3$ & $w_{63} \ra 84$ & $w_{15} \ra 20$ & $w_{51} \ra 68$ & $w_{3} \ra 4$ & $w_{39} \ra 52$ & $w_{75} \ra 100$ & $w_{27} \ra 36$ \\
    $u_4$ & $w_{28} \ra 37$ & $w_{64} \ra 85$ & $w_{16} \ra 21$ & $w_{52} \ra 69$ & $w_{4} \ra 5$ & $w_{40} \ra 53$ & $w_{76} \ra 101$ \\
    $u_5$ & $w_{77} \ra 102$ & $w_{29} \ra 38$ & $w_{65} \ra 86$ & $w_{17} \ra 22$ & $w_{53} \ra 70$ & $w_{5} \ra 6$ & $w_{41} \ra 54$ \\
    $u_6$ & $w_{42} \ra 55$ & $w_{78} \ra 103$ & $w_{30} \ra 39$ & $w_{66} \ra 87$ & $w_{18} \ra 23$ & $w_{54} \ra 71$ & $w_{6} \ra 7$ \\
    $u_7$ & $w_{7} \ra 8$ & $w_{43} \ra 56$ & $w_{79} \ra 104$ & $w_{31} \ra 40$ & $w_{67} \ra 88$ & $w_{19} \ra 24$ & $w_{55} \ra 72$ \\
    $u_8$ & $w_{56} \ra 73$ & $w_{8} \ra 9$ & $w_{44} \ra 57$ & $w_{80} \ra 105$ & $w_{32} \ra 41$ & $w_{68} \ra 89$ & $w_{20} \ra 25$ \\
    $u_9$ & $w_{21} \ra 26$ & $w_{57} \ra 74$ & $w_{9} \ra 10$ & $w_{45} \ra 58$ & $w_{81} \ra 106$ & $w_{33} \ra 42$ & $w_{69} \ra 90$ \\
    $u_{10}$ & $w_{70} \ra 91$ & $w_{22} \ra 27$ & $w_{58} \ra 75$ & $w_{10} \ra 11$ & $w_{46} \ra 59$ & $w_{82} \ra 107$ & $w_{34} \ra 43$ \\
    $u_{11}$ & $w_{35} \ra 45$ & $w_{71} \ra 93$ & $w_{23} \ra 29$ & $w_{59} \ra 77$ & $w_{11} \ra 13$ & $w_{47} \ra 61$  & $w_{83} \ra 109$ \\
    \hline
    \end{tabular}
\end{table}
\end{Example}

\section*{Concluding Remarks}

In \cite{Kim}, Kim et al. determined the radio number of a path $P_m (m \geq 4)$ and the complete graph $K_n (n \geq 3)$ as follows:
\begin{equation}\label{rn:PmKn}
\rn(P_m \Box K_n) =
\begin{cases}
\frac{m^2n-2m+2}{2}, & \mbox{if $m$ is even}, \\
\frac{m^2n-2m+n+2}{2}, & \mbox{if $m$ is odd}.
\end{cases}
\end{equation}
This result can be proved using our results. The outline of the proof is as follows: Since a path $P_m$ is a tree, we use Theorem \ref{thm:lower} and \ref{thm:ns1} to prove the result. Note that the diameter $d$ of a path $P_m$ is $m-1$. We consider the following two cases.

\textsf{Case-1:} $m$ is even. In this case, the total level of a path $P_m$ is given by $L(P_m) = m(m-2)/2$. Substituting this in \eqref{rn:lower}, we obtain $\rn(P_m \Box K_n) \geq (m^2n-2m+2)/2$. Now observe that the ordering of $V(P_m \Box K_n)$ given in \cite{Kim} satisfies the conditions (a)-(c) of Theorem \ref{thm:ns1} and hence we obtain that the first line in \eqref{rn:PmKn} holds

\textsf{Case-2:} $m$ is odd. In this case, the total level of a path $P_m$ is given by $L(P_m) = (m^2-1)/2$. Substituting this in \eqref{rn:lower}, we obtain $\rn(P_m \Box K_n) \geq (m^2n-2m+n)/2$. Now if possible then assume that $\rn(P_m \Box K_n) = (m^2n-2m+n)/2$ then by Theorem \ref{thm:ns1}, there exist an ordering $w_0,w_1,\ldots,w_{mn-1}$ of $V(P_m \Box K_n)$ such that the conditions (a)-(c) of Theorem \ref{thm:ns1} hold. Since $L(u_{i_0}) = L(u_{i_{mn-1}})$ by (a), observe that there exist a vertex $w_k = (u_{i_k},v_{j_k})$ such that $L(u_{i_k}) = (m-1)/2$ and $L(u_{i_{k-1}}) \neq 0, L(u_{i_{k+1}}) \neq 0$. Note that $d(u_{i_{k-1}},u_{i_{k+1}}) = L(u_{i_{k-1}})+L(u_{i_{k+1}})-2\phi(u_{i_{k-1}},u_{i_{k+1}})$ with $\phi(u_{i_{k-1}},u_{i_{k+1}}) \geq 1$. Denote the right-hand side of \eqref{eqn:dab} by $S_{a,b}$ and consider $u_{i_{k-1}}$ and $u_{i_{k+1}}$ in \eqref{eqn:dab} then we obtain, $S_{k-1,k+1} = L(u_{i_{k-1}})+2L(u_{i_k})+L(u_{i_{k+1}})-(d+1) = L(u_{i_k})+2((d+1)/2)+L(u_{i_{k+1}})-(d+1) = L(u_{i_{k-1}})+L(u_{i_{k+1}}) > L(u_{i_{k-1}})+L(u_{i_{k+1}})-2\phi(u_{i_{k-1}},u_{i_{k+1}}) = d(u_{i_{k-1}},u_{i_{k+1}})$ which is a contradiction. Hence, $\rn(P_m \Box K_n) \geq (m^2n-2m+n+2)/2$. Now observe that the ordering of $V(P_m \Box K_n)$ given in \cite{Kim} satisfies the conditions (a)-(c) of Theorem \ref{thm:ns1} and hence we obtain that the second  line in \eqref{rn:PmKn} holds which completes the proof.

\end{document}